\documentclass[12pt,reqno]{amsart}
\usepackage{amsthm}
\usepackage{multirow}
\usepackage{amsmath}
\usepackage[margin=1in]{geometry}
\usepackage{amssymb}
\usepackage{mathtools}
\usepackage{verbatim}
\usepackage{longtable,enumitem,tikz-cd,color}
\usepackage{stmaryrd}
\usepackage{graphicx}
\usepackage{latexsym}   
\usepackage{amsmath}    
\usepackage{amsbsy}
\usepackage{array}
\usepackage[all]{xy}
\usepackage{hyperref}
\usepackage{caption}
\usepackage{booktabs}
\usepackage{makecell}
\usepackage{adjustbox}
\usepackage{csquotes}

\theoremstyle{plain}
\newtheorem{algorithm}{Algorithm}[section]

\newtheorem{definition}[algorithm]{Definition}
\newtheorem{observe}[algorithm]{Observation}
\newtheorem{example}[algorithm]{Example}

\newtheorem{lemma}[algorithm]{Lemma}

\newtheorem{main}{Theorem}

\newtheorem*{theorem*}{Theorem}
\newtheorem{theorem} [algorithm] {Theorem}

\newtheorem{remark}[algorithm]{Remark}

\newtheorem{Codim2}[algorithm]{Codimension Two Lemma}
\numberwithin{equation}{algorithm}

\newtheorem*{1/2MSR}{Half-Maximal Symmetry Rank Theorem}
\newtheorem*{corb}{Corollary B}
\newtheorem*{main2}{Theorem C}
\newtheorem*{main1}{Theorem A}

\newtheorem*{Borel}{Borel Formula}

\newcommand{\Z}{\mathbb{Z}}
\newcommand{\R}{\mathbb{R}}\newcommand{\C}{\mathbb{C}}

\newcommand{\F}{\mathbb{F}}

\newcommand{\RP}{\mathbb{R}\mathrm{P}}
\newcommand{\CP}{\mathbb{C}\mathrm{P}}

\DeclareMathOperator{\codim}{codim}

\newcommand{\ceil}[1]{\left\lceil #1 \right\rceil}
\newcommand{\floor}[1]{\left\lfloor #1 \right\rfloor}

\usepackage[T1]{fontenc}
\usepackage{makecell}

\title{On Embedding Hamming Spheres with Applications to Positive Curvature} 
\author{Muhammad Abdullah$^{*}$}
\date{\today}

\begin{document}
\begin{abstract} We consider isometric $\Z_p$-torus actions on closed, positively curved manifolds, obtaining a homotopy classification result. This allows us to strengthen on the improvement of the $3n/8$ half-maximal symmetry rank result of Fang and Rong \cite{FR} and Ghazawneh \cite{Gh} given in Theorem B and Corollary C in the work of the author and Searle \cite{AS}. Along the way, we generalize a sphere-embedding construction of Wilking \cite{Wilk} from the binary case to arbitrary primes in Theorem \hyperref[main2]{C}, which is a geometric construction of independent interest.
\end{abstract}
\maketitle

\begingroup
\renewcommand{\thefootnote}{*}
\footnotetext{%
\begin{minipage}[t]{\dimexpr\linewidth-1.5em\relax}
Department of Mathematics, Oregon State University,
Corvallis, OR 97331, USA.\\
\textit{Email address:} \texttt{abdumuha@oregonstate.edu}
\end{minipage}}
\endgroup

\vspace{-0.6cm}
\begin{center}
\end{center}

\section{Introduction}

One approach to the classification problem for positively curved manifolds is to study those admitting \enquote{large} symmetry groups, with the aim of obtaining stronger classification results. This approach, suggested by Karsten Grove, is known as the Symmetry Program; see the survey by Grove \cite{G}. Beginning with Wilking \cite{Wilk}, techniques from the theory of error-correcting codes have played an important role in obtaining such classification results. Applications of these techniques can be found in the work of Wilking \cite{Wilk}, Kennard, Khalili Samani, and Searle \cite{KKSS}, Ghazawneh \cite{Gh,Gh1}, and the author and Searle \cite{AS}. Using finite-length Elias--Bassalygo and Plotkin bounds, the approximate $3n/8$ threshold obtained in \cite{FR,Gh} was improved in \cite{AS}. In this paper, by generalizing a sphere-embedding construction of \cite{Wilk}, we obtain a stronger generalized finite-length Elias--Bassalygo bound than that of \cite{AS}. As a consequence, we improve the symmetry rank estimates of \cite{AS} by enlarging the range of dimensions in which the corresponding classification results hold. Our first main theorem is the following.
\begin{main} \label{main1}
Assume $\Z_p^r$ acts effectively and isometrically on a closed, positively
curved manifold $M^n$ with a fixed point $x\in M$, where $p \in \{3,5,7,11,13,17,19\}$ and $n \ge 13$. 
Suppose
\begin{multline*}
r>
\frac12\Bigl(1-H_p(J_p(\tfrac14))\Bigr)n
+\frac32\log_p n
+\log_p\!\left(\frac{p-1}{2}+\frac1n\right)
-\frac12\log_p 2 \\
\qquad
+\frac12\log_p\!\left(\frac{8(1-J_p(\tfrac14))}{J_p(\tfrac14)}\right)
+\log_p(p-1)
=: G(n,p).
\end{multline*}
where $H_p$ and $J_p$ are the $p$-ary entropy and Johnson functions defined in Section \ref{2}. Then $M$ is homotopy equivalent to $S^n$, $\RP^n$, $\CP^{n/2}$, or a lens space.
\end{main}

For $3 \le p \le 19$, we have that $G(n, p)\leq c(p)n$ for $n\geq n_0(p)$. The explicit values of  $c(p)$ and $n_0(p)$ are given in Table \ref{values}. Consequently, as an immediate corollary of Theorem \ref{main1}, we obtain the following result.

\begin{table}[h]
\centering
\renewcommand{\arraystretch}{2.0}
\begin{adjustbox}{max width=\textwidth}
\begin{tabular}{|c||c|c|c|c|c|c|c|}
\hline
$p$
& 3 & 5 & 7 & 11 & 13 & 17 & 19 \\
\hline\hline
$c(p)$
& $\frac{9}{32}$
& $\frac{8}{25}$
& $\frac{17}{50}$
& $\frac{71}{200}$
& $\frac{179}{500}$
& $\frac{73}{200}$
& $\frac{37}{100}$ \\
\hline
$n_0(p)$
& 1277
& 3565
& 1777
& 3098
& 19281
& 12124
& 1976 \\
\hline
\end{tabular}
\end{adjustbox}
\caption{Values of $c(p)$ and $n_0(p)$ for $3 \le p \le 19$.}
\label{values}
\end{table}

\begin{corb} \label{corb}
    Assume $\mathbb{Z}_p^r$ acts effectively and isometrically on a closed, positively curved manifold $M^n$ with a non-empty fixed-point set, where  $p\in\{3,5,7,11,13,17,19\}$. Suppose further that $n\ge n_0(p)$ and $r\ge c(p)\,n$, where $c(p)$ and $n_0(p)$ are as in Table~\ref{values}. Then $M$ is homotopy equivalent to $S^n$, $\RP^n$, $\CP^{n/2}$, or a lens space.
\end{corb}

We immediately note that while the $c(p)$ values are the same as in Table 1 in our original paper \cite{AS}, the $n_0(p)$ values in Table \ref{values} here are, on average, approximately $29.3\%$ less than the corresponding dimensional threshold values, $n(p)$, in Table 1 in \cite{AS}. Alternatively, one may fix the dimension thresholds appearing in Table 1 of \cite{AS} and improve the corresponding symmetry rank constants, $c(p)$. For example, when $p=3$, retaining the previous dimension threshold $n(3)  = 1910$ allows the constant $c(3) = 9/32$ to be replaced by $7/25$. Thus, we have further improved the symmetry rank estimates by either enlarging the range of dimensions for which the homotopy classification results hold or, with the previous dimension thresholds fixed, lowering the corresponding symmetry rank constants. Moreover, we also later show, in the discussion preceding the proof of Theorem \ref{newEB}, that the estimate in Theorem \ref{main1} is an improvement over the estimate in Theorem B in \cite{AS} by a factor of $\approx \log_pn$. We refer the reader to Sections \ref{2} and \ref{4} for more detail.

We write $|S|$ for the cardinality of a set $S$. With this convention, our second main theorem is as follows.

\begin{main2} \label{main2}
Let $p\ge2$ and let $S_w(0)\subseteq\mathbb Z_p^m$ be a set of vectors of Hamming weight $w$. There exists an embedding
$f:S_w(0)\hookrightarrow S^{(p-1)m-1}\subset\mathbb R^{(p-1)m}$
such that if $b\le m$ is an integer satisfying
$\frac{w}{m}<J_p\!\left(\frac{b}{m}\right)$
and $\mathcal C$ is any subset of $S_w(0)$ that satisfies $d(x,y)\ge b$ for all distinct $x,y\in\mathcal C$, then

\begin{enumerate}
\item For all distinct $x,y\in\mathcal C$, $\langle f(x),f(y)\rangle<0$, where $\langle \,\cdot\,,\,\cdot\,\rangle$ denotes the Euclidean inner product; and
\item $|\mathcal C|\le\min\left\{(p-1)m+1,\gamma(p,b,w,m)\right\}$,
where \begin{equation}\label{gamma}
\gamma(p,b,w,m)
:=
\floor{\frac{b}{\,b-mJ_p^{-1}(\frac{w}{m})\,}}.
\end{equation}.
\end{enumerate}
\end{main2}
We first note that this theorem is the only result in this paper in which $p$ is allowed to be an arbitrary positive integer, rather than a prime. Since our primary interest is in linear codes over finite fields, in all subsequent applications, $p$ will denote a prime number. We also remark that a version of Theorem \hyperref[main2]{C} was proved in \cite{Wilk} for the case $p=2$ as part of the proof of Proposition 3.1 in \cite{Wilk}, and was later stated in a slightly different form as Lemma 1.3 in \cite{KKSS}. Our formulation follows that of \cite{KKSS}. The construction, detailed in Section \ref{2}, embeds a Hamming sphere $S_w(0)\subset\Z_p^m$ into $S^{(p-1)m-1}\subset\R^{(p-1)m}$ in such a way that the distance assumptions of Theorem \hyperref[main2]{C} force the image vectors to be pairwise obtuse. A Euclidean sphere-packing argument then yields a bound on the cardinality of $\mathcal C\subseteq S_w(0)$. Decomposing a $p$-ary code by intersecting it with Hamming spheres and applying this estimate, we obtain an upper bound on the corresponding information rate of the code. For codes arising from isotropy representations at fixed points, this translates directly into an estimate on the symmetry rank $r$.

The embedding of Theorem \hyperref[main2]{C} is inspired by both the work of \cite{Wilk}, which it generalizes and by Lemma 1.1 in Rudra \cite{Ru3}, which is used there in the proof of the Johnson bound. Such discrete embeddings are useful in the theory of error-correcting codes because one can often \enquote{slice} a general code by Hamming spheres and use bounds on the cardinality of each \enquote{slice} to obtain a bound on the information rate of the code. However, the embedding of Lemma 1.1 in \cite{Ru3} does not take the constant-weight structure of the code into account. Consequently, the image vectors are obtuse only for sufficiently large values of $b$, so the Euclidean sphere-packing argument used in \cite{Wilk} and \cite{KKSS} cannot be applied for the smaller values of $b$ that we require to prove Theorem \ref{main1} and Corollary \hyperref[corb]{B}. In the embedding constructed in Theorem \hyperref[main2]{C}, we instead exploit the structure of a regular simplex to obtain an optimal configuration of pairwise obtuse vectors, allowing us to work with the smaller values of $b$ required by our geometric applications. To make the optimality of this construction more apparent, we provide an example in Lemma \ref{projectivegamma} demonstrating the sharpness of Part (2) of Theorem \hyperref[main2]{C}.

Having outlined the main construction of this paper, we additionally remark that, for the proof of Theorem \hyperref[main1]{A}, the assumption of a fixed point is necessary, since the techniques used in our proofs rely on the existence of a fixed-point set component and the induced $\Z_p$-torus action on such a component. In general, unlike the case of isometric torus actions in positive curvature, where Berger’s theorem \cite{Be} guarantees a fixed point in even dimensions and Sugahara’s theorem \cite{Su} guarantees a fixed point for a codimension-one subtorus in odd dimensions, we know of no analogous results for $\Z_p^r$-actions, with the exception of even dimensions and primes strictly greater than $b(n)$, the constant guaranteed by work of Gromov \cite{Gr}; see Lemma 2.1 in \cite{FR}. We refer the reader to \cite{KKSS} and \cite{AS} for a further discussion. Finally, to the best of the author's knowledge, the bounds in Theorem \ref{main1} and Corollary \hyperref[corb]{B} are the strongest currently available for $\Z_p^r$-actions when $3\le p\le19$.
\subsection{Organization}

The paper is organized as follows. In Section \ref{2}, we recall some notions from the theory of error-correcting codes, prove Theorem \hyperref[main2]{C}, and use it to derive a finite-length generalized $p$-ary Elias-Bassalygo bound that improves the bound in Theorem 2.4 of \cite{AS}. In Section \ref{3}, we recall the geometric and topological preliminaries needed to prove Theorem \hyperref[main1]{A}. Finally, in Section \ref{4}, we conclude by proving Theorem \ref{main1}.

\subsection{Acknowledgments}
The author is grateful to his advisor, Catherine Searle, for her continued guidance and support. The author is also grateful to Lee Kennard for helpful conversations. This work formed a part of the author's master's thesis at Wichita State University. 

\section{The Proof of Theorem C}\label{2}

Throughout this section, we assume $p$ is prime unless stated otherwise. Before proving Theorem C, we recall several standard notions from the theory of error-correcting codes and make explicit the connection between the theory of error-correcting codes and the geometry of positive curvature.

A $p$-ary linear code of length $m$ is a linear subspace $\mathcal C\subseteq\F_p^m$, where $\F_p\cong\Z_p$ is the finite field with $p$ elements. Since our applications arise from effective actions of elementary abelian $p$-groups, we restrict attention to codes over $\F_p$. However, the arguments below apply equally well to arbitrary subsets of $\F_p^m$, and thus do not require linearity. For simplicity, we shall use the term \emph{code} to mean a $p$-ary linear code.

\begin{definition}
Let $x\in\F_p^m$. The \emph{\bf Hamming weight} of $x$, denoted $|x|$, is the number of nonzero entries of $x$. The \emph{\bf Hamming distance} between $x,y\in\F_p^m$ is $|x-y|$, which we will sometimes denote as $d(x,y)$.
\end{definition}

The error-detecting and error-correcting capabilities of a code are governed by its minimum distance.

\begin{definition}
For a $p$-ary code $\mathcal C$ of dimension $k$ and length $m$, the \emph{\bf minimum distance} of $\mathcal C$ is
$$
d_{\mathcal C}=\min\{|x-y|:x,y\in\mathcal C,\ x\neq y\}.
$$
If $d_{\mathcal C}\ge d$, then $\mathcal C$ is called an $[m,k,d]_p$-code.
\end{definition}

A central problem in coding theory is to determine how efficiently information can be transmitted through a noisy channel. This leads to the study of
$$
A_p(m,d)=\max\{|\mathcal C|:\mathcal C\subseteq\F_p^m,\ d_{\mathcal C}\ge d,\}
$$
and the information rate
$$
R_p(m,d)=\frac1m\log_pA_p(m,d).
$$
While much of the error-correcting code literature focuses on asymptotic bounds for the information rate, our interest lies in explicit finite-length estimates. Among the classical asymptotic bounds, including the Singleton, Hamming, Plotkin, Elias--Bassalygo, and McEliece--Rodemich--Rumsey--Welch bounds, the Elias--Bassalygo bound is strongest for small prime alphabets $3\le p\le19$; see \cite{vanLint1999}. Motivated by this observation, we develop a stronger finite-length Elias--Bassalygo bound by generalizing the construction  in \cite{Wilk}. Previous finite-length versions appear in \cite{AS}, although for $p=2$, the bound in \cite{Wilk} remains stronger. Our generalization recovers the estimate in \cite{Wilk} for $p=2$ and improves the corresponding bound in \cite{AS} for all primes.

As indicated at the start of this section, we now make explicit the connection between the theory of error-correcting codes and the geometry of positive curvature. Let $M^n$ be a positively curved manifold admitting an effective, isometric $\Z_p^r$-action, let $\tau\in\Z_p^r$ be nontrivial, and let $N$ be a component of $M^\tau$, the fixed point set in $M$ of $\tau$. For $x\in N$, the isotropy representation $\rho_x(\tau)=d\tau_x$ acts trivially on $T_xN$ and non-trivially on the normal space $\nu_xN$. Since $\Z_p^r$ is abelian, the isotropy representation is simultaneously diagonalizable over $\C$. For $p>2$, the action on $\nu_xN$ decomposes into $2\times2$ rotation blocks, yielding an injective homomorphism
$$
\phi_x:\Z_p^r\to\Z_p^m,
$$
where $m=\lfloor n/2\rfloor$. The effectiveness of the action implies that $\phi_x$ is injective, and fixed-point codimensions are encoded by Hamming weights via
$$
\codim(N)=2|\phi_x(\tau)|.
$$
As we will later see, geometric questions concerning fixed-point set components of low codimension translate into extremal problems for $p$-ary codes, allowing coding-theoretic bounds to yield symmetry rank estimates.

\subsection{The Embedding Lemma} \label{2.1}

Before proving Theorem \hyperref[main2]{C}, we recall the definitions of some standard functions from the theory of error-correcting codes and a geometric lemma.

\begin{definition}[{${p}$-ary entropy function and the Johnson radius}] \label{entropy} Let $\delta \in [0,\frac{p-1}{p}]$.
The \emph{\bf $\mathbf{p}$-ary entropy function} $H_p(x)$ and the \emph{\bf Johnson radius} $J_p(\delta)$ are respectively defined by
\begin{align*}
H_p(x) &= x \log_p(p-1) - x \log_p x - (1 - x) \log_p(1 - x), \,\,\text{and}\\
J_p(\delta) &= \left(1 - \frac{1}{p} \right) \left(1 - \sqrt{1 - \frac{p}{p-1} \delta} \right).
\end{align*}
\end{definition}

\begin{remark} \label{remarkJohnson&Entropy}
The function $J_p$ is an increasing function that maps the interval $[0,\frac{p-1}{p}]$ onto itself, and the function $\epsilon \mapsto 1-H_p(\epsilon)$ decreases in the same interval.
\end{remark}

We now collect two useful sphere-packing results from linear algebra into a lemma. Part (1) follows from Theorem 1 of Rankin \cite{Rankin}, while Part (2) appears in \cite{Ru3}. The latter is also implicit in Rankin's theorem: under the correspondence $\varepsilon=-\cos(2\alpha)$, Rankin's angular bound becomes $1+1/\varepsilon$, with the transition at $\varepsilon=1/d$.

\begin{lemma}\label{pack}
Let $d \ge 1$.
\begin{enumerate}
\item \cite{Rankin} Any collection of $d+2$ vectors in $\mathbb{R}^d$ contains two whose angle is non-obtuse; equivalently, if $v_1,\dots,v_{d+2}\in\mathbb{R}^d$, then $v_i\cdot v_j\ge 0$ for some $i\ne j$.

\item \cite{Ru3} If $v_1,\dots,v_N\in S^{d-1}\subseteq\mathbb R^d$ satisfy $v_i\cdot v_j\le-\varepsilon$ for all $i\neq j$, where $\varepsilon>0$, then $N\le \min\{d+1, \floor{1+\frac1\varepsilon}\}$.
\end{enumerate}
\end{lemma}

\begin{proof}
We sketch the proof for the convenience of the reader. Part (1) follows by a standard induction on $d$. For Part (2), the estimate $$0\le\left\|\sum_{i=1}^N v_i\right\|^2\le N-\varepsilon N(N-1)$$ gives $N\le1+\frac1\varepsilon$. Combining this with Part (1) yields $N\le\min\{d+1, \floor{1+\frac1\varepsilon}\}$, where we can replace $1+\frac1\varepsilon$ by $\floor{1+\frac1\varepsilon}$ since $N$ is a natural number.
\end{proof}

Before recalling Theorem \hyperref[main2]{C}, we remind the reader that, as noted in the Introduction, this is the only result in the paper for which $p$ is allowed to be an arbitrary positive integer. In all subsequent applications, $p$ again denotes a prime. We now restate Theorem \hyperref[main2]{C} for the reader's convenience and proceed with its proof after a clarifying remark.

\begin{main2}
Let $p\ge2$ and let $S_w(0)\subseteq\mathbb Z_p^m$ be a set of vectors of Hamming weight $w$. There exists an embedding
$f:S_w(0)\hookrightarrow S^{(p-1)m-1}\subset\mathbb R^{(p-1)m}$
such that if $b\le m$ is an integer satisfying
$\frac{w}{m}<J_p\!\left(\frac{b}{m}\right)$
and $\mathcal C$ is any subset of $S_w(0)$ that satisfies $d(x,y)\ge b$ for all distinct $x,y\in\mathcal C$, then

\begin{enumerate}
\item For all distinct $x,y\in\mathcal C$, $\langle f(x),f(y)\rangle<0$, where $\langle \,\cdot\,,\,\cdot\,\rangle$ denotes the Euclidean inner product; and
\item $|\mathcal C|\le\min\left\{(p-1)m+1,\gamma(p,b,w,m)\right\}$,
where $\gamma(p,b,w,m)$ is defined in Display \ref{gamma}.
\end{enumerate}
\end{main2}

\begin{proof}
Assume, without loss of generality, that $p \ge 3$, since the construction below involves division by $p-2$. We refer the reader to \cite{Wilk} and \cite{KKSS} for the case $p=2$, noting that the underlying construction is essentially the same as the one presented here, but the corresponding formulas simplify and require a separate normalization. 

Choose vectors $u_a\in\R^{p-2}$ for $a\in\Z_p\setminus\{0\}$ forming a regular simplex, so that $\langle u_a,u_a\rangle=1$ and $\langle u_a,u_b\rangle=-\frac1{p-2}$ whenever $a\neq b$. Identifying $\R^{p-2}$ with the hyperplane $H=\{x\in\R^{p-1}:\sum_{i=1}^{p-1}x_i=0\}$, one may take
$$
u_i=\sqrt{\frac{p-1}{p-2}}\left(e_i-\frac1{p-1}\mathbf 1\right).
$$
Let $e_0$ be a unit vector orthogonal to $H$, define $\alpha=\frac1w$, $\beta=-\frac1{m-w}$, and $\lambda=(\alpha-\beta)\sqrt{\frac{p-2}{p}}$, and set $v_0=\beta e_0$ and $v_a=\alpha e_0+\lambda u_a$ for $a\in\Z_p\setminus\{0\}$. See Figure \hyperref[visualC]{1} for the cases $p=3$ and $p=4$.

\begin{figure}[ht] \label{visualC}
\centering

\includegraphics[width=0.42\textwidth]{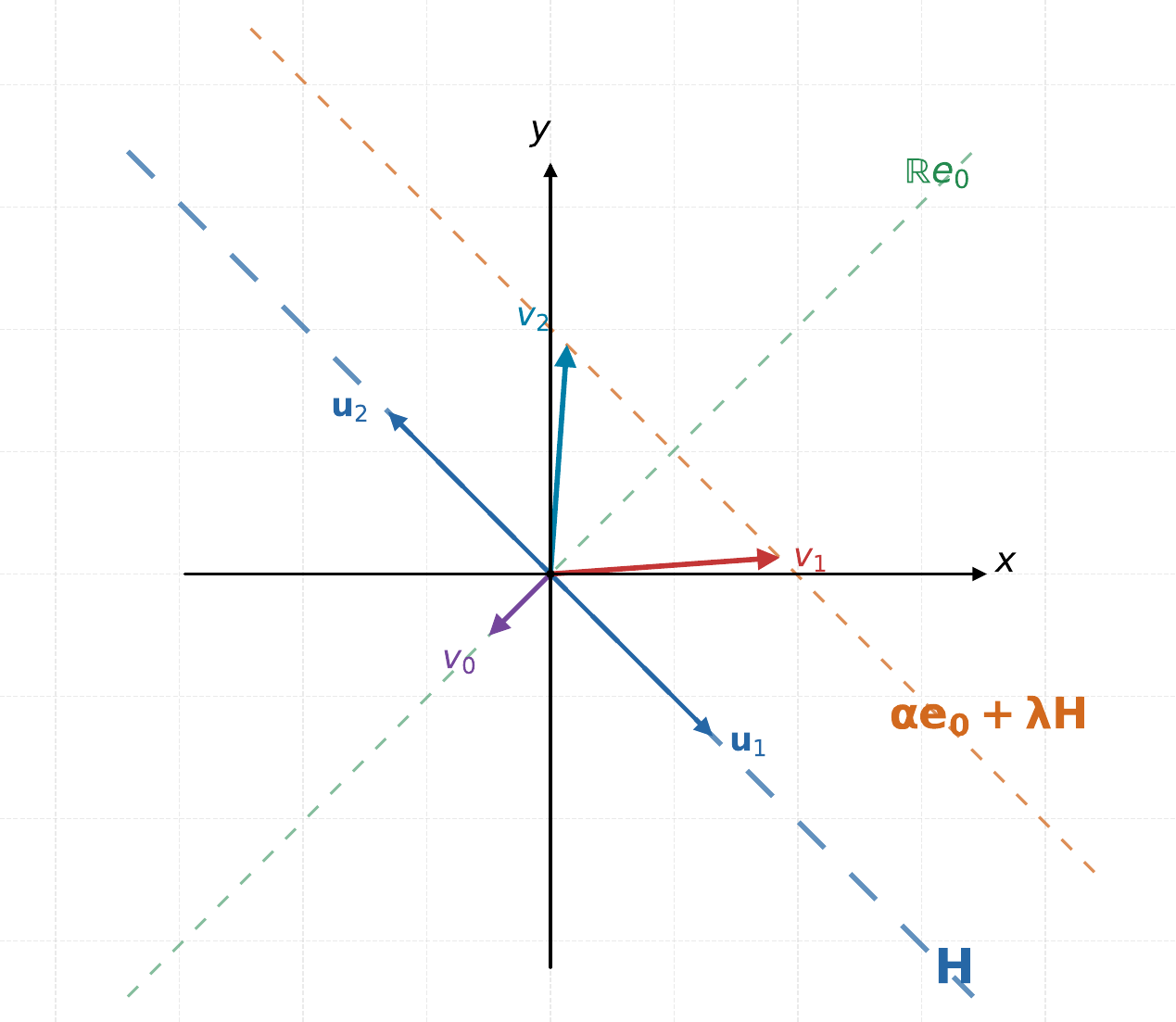}
\hfill
\includegraphics[width=0.48\textwidth]{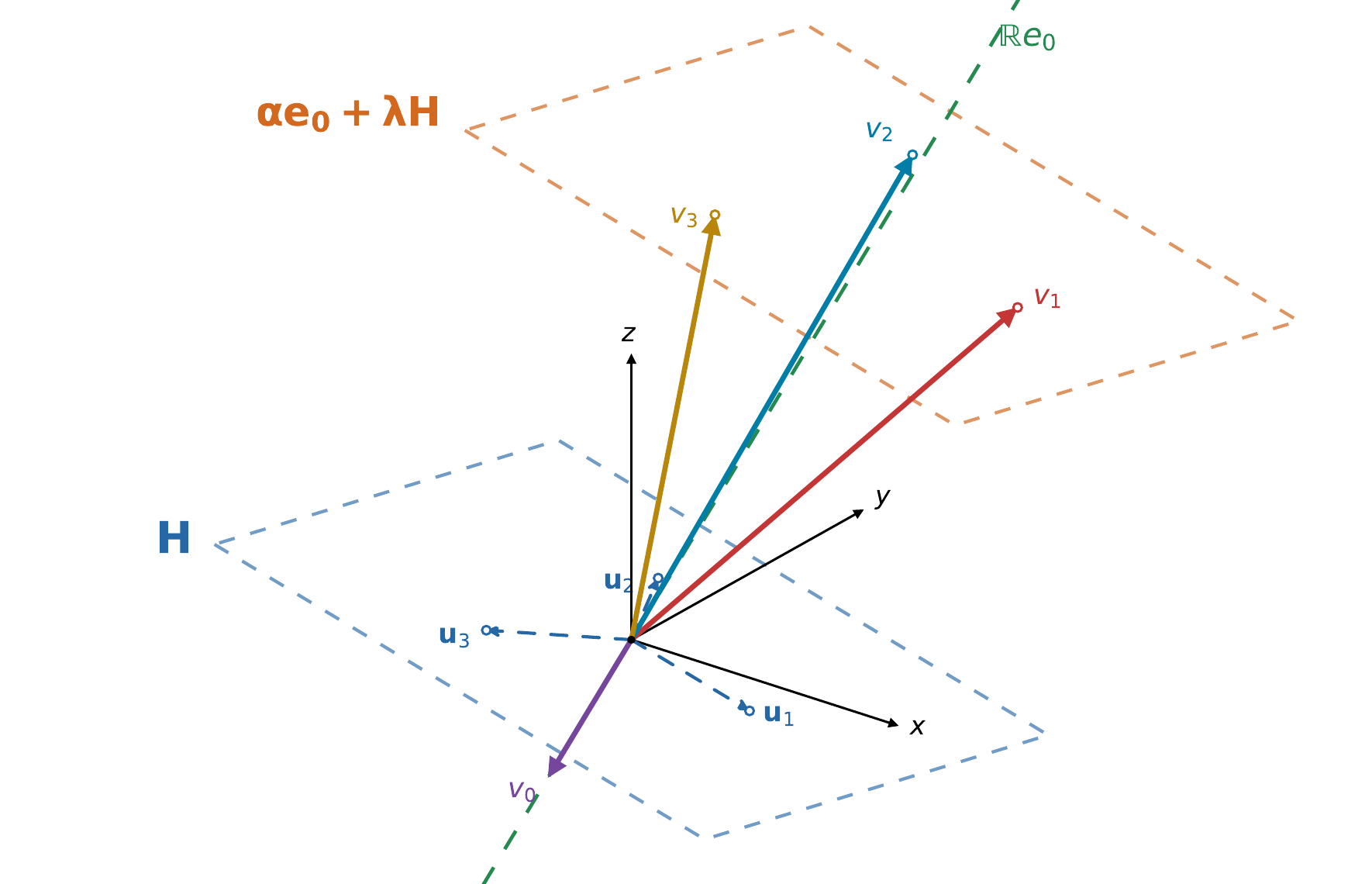}

\caption{
The vectors used in the proof of Theorem \hyperref[main2]{C} for $p=3$ and $p=4$.
}
\end{figure}

For $x=(x_1,\dots,x_m)\in S_w(0)$, define $F(x)=(v_{x_1},\dots,v_{x_m})\in\R^{(p-1)m}$. Write $A=\langle v_a,v_a\rangle$, $B=\langle v_a,v_b\rangle$ $(a\neq b)$, $E=\langle v_0,v_a\rangle$, and $D=\langle v_0,v_0\rangle$. A direct calculation gives $A=\alpha^2+\lambda^2$, $B=\alpha^2-\frac{\lambda^2}{p-2}$, $E=\alpha\beta$, $D=\beta^2$, and $B=E+\frac{A-D}{2}$.

Fix distinct $x,y\in S_w(0)$. Let $s$ be the number of coordinates with $x_i=y_i\neq0$, let $t$ be the number with $x_i\neq y_i$ and both entries nonzero, and let $k$ and $\ell$ denote the numbers of coordinates of type $(\ast,0)$ and $(0,\ast)$, respectively. Since both vectors have weight $w$, we have $k=\ell$ and $s+t+k=w$, while $d(x,y)=t+2k$. Expanding the inner product gives
\begin{equation}\label{eq:inner}
\langle F(x),F(y)\rangle =
wA+(m-w)D-\left(\frac A2-E+\frac D2\right)d(x,y).
\end{equation}
Since $|F(x)|^2=wA+(m-w)D$, the map
$$
f(x)=\frac{F(x)}{\sqrt{wA+(m-w)D}}
$$
takes values in $S^{(p-1)m-1}$. Using Equation \eqref{eq:inner},
\begin{equation}\label{eq:sum}
\langle f(x),f(y)\rangle
=
1-\frac{d(x,y)}
{\,2w-\frac{p}{p-1}\frac{w^2}{m}\,}.
\end{equation}
where we have used the identity
$$
\frac{wA+(m-w)D}{\frac A2-E+\frac D2} = 
2w-\frac{p}{p-1}\frac{w^2}{m}.
$$
Writing $\omega=w/m$ and $\delta=b/m$, the hypothesis $\omega<J_p(\delta)$, implies that $$\delta>J_p^{-1}(\omega)=2\omega-\frac{p}{p-1}\omega^2,$$ and hence
\begin{equation} \label{ineqb}
b>2w-\frac{p}{p-1}\frac{w^2}{m}.    
\end{equation}
Using the assumption that $d(x,y)\ge b$ for distinct $x, y \in \mathcal C$, Equation \eqref{eq:sum} shows that $\langle f(x),f(y)\rangle<0$ whenever $x \ne y$ and $x,y \in \mathcal{C}$, proving Part (1) of our Theorem.

Finally, $\{f(x):x\in \mathcal{C}\}$ is a collection of unit vectors in $\R^{(p-1)m}$ with pairwise negative inner products. Setting $$\varepsilon=-\left(1-\frac{b}{2w-\frac{p}{p-1}\frac{w^2}{m}}\right)=-\left(1-\frac{b}{mJ_p^{-1}(w/m)}\right),$$ we see from Equation \ref{eq:sum} that the hypotheses of Lemma \ref{pack} are satisfied. Then, by Inequality \ref{ineqb}, we see that
$$
|\mathcal{C}|
\le
\min\left\{(p-1)m+1,\gamma(p,b,w,m)\right\},
$$
proving Part (2) of the theorem.
\end{proof}

We now give a necessary and sufficient condition for when $\gamma(p,b,w,m)$ is stronger than $(p-1)m+1$, followed by an analysis of its utility.

\begin{lemma} \label{boundcriterion}
The bound $\gamma(p,b,w,m)$ is stronger than $(p-1)m+1$ if and only if
$$
w<mJ_p\!\left(\frac{(p-1)b}{(p-1)m+1}\right).
$$
Thus, the only values of $w$ for which the bound $(p-1)m+1$ can be sharper must satisfy
$$
mJ_p\!\left(\frac{(p-1)b}{(p-1)m+1}\right)\le w<mJ_p\!\left(\frac{b}{m}\right).
$$
\end{lemma}

\begin{proof}
By definition, $$\gamma(p,b,w,m)=\left\lfloor\frac{b}{b-mJ_p^{-1}(\frac{w}{m})}\right\rfloor.$$ Since $(p-1)m+1$ is an integer, $$\gamma(p,b,w,m)<(p-1)m+1$$ if and only if
$$\frac{b}{b-mJ_p^{-1}(\frac{w}{m})}<(p-1)m+1.$$ Since $b-mJ_p^{-1}(\frac{w}{m})>0$, this is equivalent to
$b<((p-1)m+1)(b-mJ_p^{-1}(\frac{w}{m}))$, or equivalently,
$$J_p^{-1}(\frac{w}{m})<\frac{(p-1)b}{(p-1)m+1}.$$
Since $J_p$ is strictly increasing on $[0,\frac{p-1}{p}]$, this is equivalent to
$$w<mJ_p\!\left(\frac{(p-1)b}{(p-1)m+1}\right),$$ proving the first statement. The second statement follows immediately from the contrapositive together with the hypothesis
$w<mJ_p\!\left(\frac bm\right)$.
\end{proof}

\begin{remark}
In general, it is difficult to give a more transparent criterion than the one in Lemma \ref{boundcriterion} for when $\gamma(p,b,w,m)$ is stronger than $(p-1)m+1$, as the comparison depends on the relationship between the parameters $b$ and $w$. However, as we see in Lemma \ref{projectivegamma} below, $\gamma(p,b,w,m)$ is indeed sometimes stronger than $(p-1)m+1$.
\end{remark}

As mentioned in the Introduction, we now demonstrate the sharpness of Part (2) of Theorem~\hyperref[main2]{C} for certain codes. It turns out that the appropriate family of codes for this purpose is the family of \emph{projective codes}. Recall that the points of the finite projective space $\mathrm{PG}(k-1,p)$ are the one-dimensional subspaces of $\Z_p^k$. Since $\mathrm{PG}(k-1,p)$ has $m=(p^k-1)/(p-1)$ points, we may index the coordinates of $\Z_p^m$ by these points and choose a nonzero representative $v_i$ for each. Setting $V=\{v_i\}_{i=1}^m$, we define the associated {\it projective code $\mathcal C_V$} by
$$
\mathcal C_V\coloneqq\left\{(u\cdot v_1,\ldots,u\cdot v_m):u\in\Z_p^k\setminus\{0\}\right\}.
$$
We emphasize that the subset $\mathcal C$ in Theorem \hyperref[main2]{C} is not assumed to be linear; in particular, $\mathcal C_V$ as defined here is not a linear code. The zero coordinates of $x_u\coloneqq(u\cdot v_1,\ldots,u\cdot v_m) \in \mathcal{C}_{V}$ are indexed precisely by the projective points contained in the hyperplane $u^\perp$ in $\Z_p^k$. Although the values of the nonzero coordinates depend on the chosen representatives $v_i$, the locations of the zero coordinates, and hence the weight of $x_u$, do not.

\begin{lemma}\label{projectivegamma}
Let $p$ be a prime and $k\ge2$, and set $m=(p^k-1)/(p-1)$. Let $\mathcal C_V$ be a projective code associated to $\mathrm{PG}(k-1,p)$. Then every $x\in\mathcal C_V$ has weight $w=p^{k-1}$, any two distinct $x,y\in\mathcal C_V$ satisfy $d(x,y)=b=p^{k-1}$, and
$$
|\mathcal C_V|=\gamma(p,b,w,m)=p^k-1.
$$
In particular,
$$
\gamma(p,b,w,m)<(p-1)m+1.
$$
\end{lemma}
\begin{proof}
The map $u\mapsto x_u$ is injective. Indeed, if $x_u=0$, then $u\cdot v_i=0$ for every $i$. Since the vectors $v_i$ represent all one-dimensional subspaces of $\Z_p^k$, this implies that $u=0$. Hence $|\mathcal C_V|=p^k-1$.

The zero coordinates of $x_u$ are indexed by the one-dimensional subspaces contained in $u^\perp$. Since $u^\perp$ has dimension $k-1$, it contains $(p^{k-1}-1)/(p-1)$ one-dimensional subspaces. Thus $|x_u|=p^{k-1}$. Moreover, if $u\ne u'$, then $x_u-x_{u'}=x_{u-u'}$, so
$d(x_u,x_{u'})=|x_{u-u'}|=p^{k-1}.$ Thus $b=w=p^{k-1}$. Moreover, since $mJ_p^{-1}(\frac{w}{m})=\frac{p^{k-1}(p^k-2)}{p^k-1},$
we obtain
$$
\gamma(p,b,w,m)
=
\floor{\frac{b}{b-mJ_p^{-1}(\frac{w}{m})}}
=
p^k-1
=
|\mathcal C_V|.
$$
Finally, $(p-1)m+1=p^k$, so $\gamma(p,b,w,m)<(p-1)m+1$.
\end{proof}

We illustrate the sharpness established in Lemma~\ref{projectivegamma} in the simplest nontrivial case. This example also provides a concrete realization of the projective codes appearing in the lemma.

\begin{example}\label{projectivecode}
For $p=2$ and $k=3$, the projective space $\mathrm{PG}(2,2)$ is the Fano plane. Its seven points are the one-dimensional subspaces of $\mathbb{F}_2^3$, so they index the seven coordinates of the projective code $\mathcal C_V$, where $V$ is the set of representatives of the points of $\mathrm{PG}(2,2)$. Each projective line consists of three points, corresponding to the one-dimensional subspaces contained in a hyperplane of $\mathbb{F}_2^3$. Consequently, for any nonzero $u\in\mathbb{F}_2^3$, the zero coordinates of the codeword $x_u$ are precisely the three points lying on the projective line corresponding to the hyperplane $u^\perp$. Figure~\ref{fig:fano} illustrates one such codeword: the points on the thickened projective line correspond to the zero coordinates of $x_u$, while the remaining points correspond to its nonzero coordinates.

By Lemma~\ref{projectivegamma},
$
|\mathcal C_{V}|=\gamma(2,4,4,7)=7<8=(2-1)7+1,
$
so the bound $\gamma(2,4,4,7)$ is attained.
\end{example}

\begin{figure}[h] 
\centering
\includegraphics[width=0.45\textwidth]{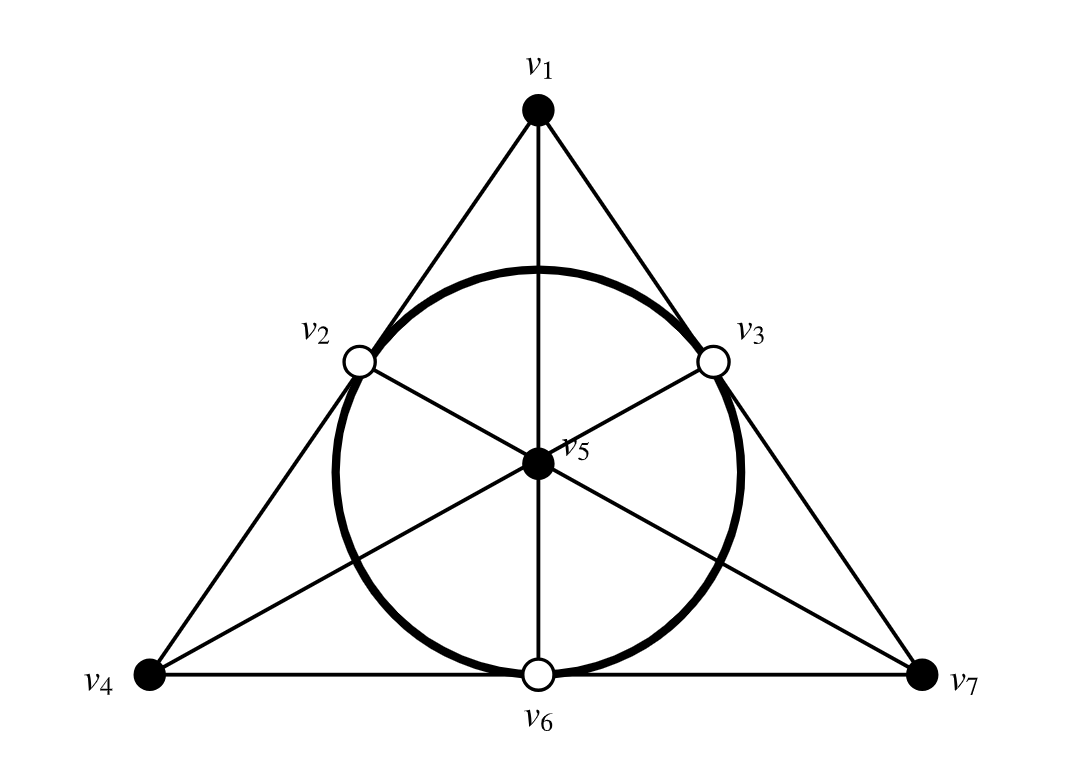}
\caption{The projective plane $\mathrm{PG}(2,2)$ over $\mathbb{F}_2$.}
\label{fig:fano}
\end{figure}

\subsection{A New Finite-Length Generalized Elias-Bassalygo Bound}

Having proved Theorem \hyperref[main2]{C}, we are now in a position to derive the following finite-length generalized Elias--Bassalygo bound, which we later use in the proof of Theorem \hyperref[main1]{A}.

\begin{theorem} \label{newEB}
Fix $\delta \in (0,\frac{p-1}{p})$. Suppose
$\rho : \mathbb{Z}_p^r \rightarrow \mathbb{Z}_p^m$
is injective and its image $C$ has minimum distance at least $\delta m$. Then
$$
r \le (1 - H_p(J_p(\delta)))m
+ \frac{1}{2}\log_p m
+ \frac{1}{2}\log_p\!\left(\frac{8(1-J_p(\delta))}{J_p(\delta)}\right)
+ \log_p(p-1)
+ \log_p\bigl(\kappa_p(\delta,m)\bigr),
$$
where
$\kappa_p(\delta,m) :=
\min\left\{(p-1)m+1,\gamma\!\left(p,\delta m,\lceil J_p(\delta)m\rceil-1,m\right)\right\}$.
\end{theorem}

Before proving the theorem, it is worth comparing this estimate with previous bounds. For $p=2$, the bound is only slightly weaker than Theorem~1.2 of \cite{KKSS}, differing by at most an additive constant of $1$. This difference arises from the estimate $m-w+1\le m$ used in the proof below, which is slightly coarser than the corresponding estimate in \cite{KKSS}. Since $\log_p((p-1)m+1)\sim\log_p m$ as $m\to\infty$, the present result improves the logarithmic term of the Elias--Bassalygo bound in Corollary~2.6 of \cite{AS} from $2.5\log_p m$ to $1.5\log_p m$, yielding a gain of order $\log_p m$. Moreover, the argument is more direct, avoiding the continuous relaxation and remainder-term estimates required in the proof of Corollary~2.6 in \cite{AS}, and consequently leads to a cleaner final expression. Keeping these remarks in mind, we prove Theorem \ref{newEB}.

\begin{proof}[Proof of Theorem \ref{newEB}]

This is a similar proof to that of Theorem 1.2 in \cite{KKSS}. Set $b=\delta m$ and define $w := \ceil{J_p(\delta)m}$ and $w_0 := w-1$. Then $w_0/m < J_p(b/m)$. Let $C=\rho(\Z_p^r)\subseteq \Z_p^m$ and consider the set
\[
\mathcal{A}
=
\{(c,x)\in C\times \Z_p^m : d(c,x)=w_0\}.
\]
For each fixed $c\in C$, the number of $x\in \Z_p^m$ with $d(c,x)=w_0$ is $\binom{m}{w_0}(p-1)^{w_0}$, since one chooses $w_0$ coordinates in which $x$ differs from $c$, and in each such coordinate there are $(p-1)$ choices. Summing over $c\in C$, we obtain
\[
|\mathcal{A}| = p^r \binom{m}{w_0}(p-1)^{w_0}.
\]
On the other hand, fix $x\in \Z_p^m$ and consider the set
\[
C_x := \{c\in C : d(c,x)=w_0\}.
\]
Define $\phi_x : C_x \to \Z_p^m$ by $\phi_x(c)=c-x$. This map is a bijection onto its image, and for each $c\in C_x$ we have $|\phi_x(c)| = d(c,x)=w_0$, so $\phi_x(C_x)$ consists of vectors of weight $w_0$. Moreover, for $c_1,c_2\in C_x$, $d(\phi_x(c_1),\phi_x(c_2)) \ge b$, since $C$ has minimum distance at least $b$. Thus $\phi_x(C_x)$ is a set of vectors of weight $w_0$ with pairwise distance at least $b$, and Theorem \hyperref[main2]{C} implies that $|C_x|\le \kappa_p(\delta,m)$. Summing over all $x\in \Z_p^m$, we obtain $|\mathcal{A}| \le p^m\kappa_p(\delta,m).$
Comparing the two expressions for $|\mathcal{A}|$ yields
\[
p^r \binom{m}{w_0}(p-1)^{w_0}
\le
p^m\kappa_p(\delta,m).
\]
Using the identity
\[
\binom{m}{w_0}
=
\binom{m}{w}\frac{w}{m-w+1},
\]
we rewrite this as
\[
p^r \binom{m}{w}(p-1)^w
\le
p^m (p-1)\frac{m-w+1}{w}\kappa_p(\delta,m).
\]

We first assume that $2 \le w \le \frac{p-1}{p}m$, and treat the remaining cases separately below. Set $\epsilon=w/m$. By Lemma 4.7.1 in Ash \cite{Ash}, which comes from Stirling's estimate,
\[
\binom{m}{w}(p-1)^w
\ge
\frac{p^{H_p(\epsilon)m}}{\sqrt{8\epsilon(1-\epsilon)m}}.
\]
Substituting the preceding estimate and using $m-w+1\le m$, we obtain
\[
p^r
\le
p^{(1-H_p(\epsilon))m}
\cdot
(p-1)\frac{m}{w}\kappa_p(\delta,m)
\sqrt{8\epsilon(1-\epsilon)m}.
\]
Taking $\log_p$ of the above inequality, using $w=\epsilon m$, and rearranging terms, we obtain
\[
r \le (1 - H_p(\epsilon))m
+ \frac{1}{2}\log_p m
+ \frac{1}{2}\log_p\!\left(\frac{8(1-\epsilon)}{\epsilon}\right)
+ \log_p(p-1)
+ \log_p\bigl(\kappa_p(\delta,m)\bigr).
\]
Since $H_p$ is increasing on $\left(0,\frac{p-1}{p}\right)$ and $(1-\epsilon)/\epsilon$ is decreasing on this interval, the right-hand side is decreasing in $\epsilon$. As $\epsilon \ge J_p(\delta)$, we substitute $\epsilon=J_p(\delta)$.

If $w<2$, then $w_0=0$, so for each $x$ the set $C_x$ contains at most one element, and hence $|C|=1$, implying $r=0$.

Finally, suppose $w > \frac{p-1}{p}m$. Then $w=\lceil J_p(\delta)m\rceil > \frac{p-1}{p}m$, so $J_p(\delta)m > \frac{p-1}{p}m - 1$. Using the formula for $J_p$ in Definition \ref{entropy}, this implies that
$$
\delta > \frac{p-1}{p} - \frac{p}{(p-1)m^2}.
$$
Multiplying by $m$ gives
$$
b > \frac{p-1}{p}m - \frac{p}{(p-1)m}
> \frac{p-1}{p}m - 1.
$$
Since $b$ is an integer, it follows that $b \ge \left\lceil \frac{p-1}{p}m \right\rceil$, and hence $\delta=b/m \ge \frac{p-1}{p}$, contradicting the hypothesis.
\end{proof}

\section{Topological and Geometric Preliminaries}\label{3}

We now recall the geometric and topological results required for the proofs of Theorem \hyperref[main1]{A} and Corollary \hyperref[main1]{B}. The principal tools used in proving these results are the Connectedness Lemma in \cite{Wilk} and several recent developments in \cite{KKSS}. Detailed proofs of the results presented below may be found in the cited references and are omitted here for brevity.

The first result we recall is a weaker version of the $\frac{3}{4}$-Maximal Symmetry Rank result of \cite{FR} and \cite{Gh}. For an outline of a proof of this weaker formulation, we refer the reader to \cite{AS}. We use this result later to anchor the induction in the proof of Theorem \hyperref[main1]{A}. 

\begin{theorem}\cite{FR, Gh} \label{Gh+Zp} 
Let $M^n$, $n\geq 13$ be a closed Riemannian manifold of positive sectional curvature. Suppose that $M$ admits an isometric and effective $\Z_p^r$-action with a fixed point $x\in M$, $p>2$. If
$$
r\,\geq \begin{cases}
   \lfloor \frac{3n}{8}\rfloor +2 \,\, \textrm{if }\,\, n\equiv 2, 4 \mod{8}\\
   \lfloor \frac{3n}{8}\rfloor +1 \,\, \textrm{, }  
\end{cases}
$$
then $M$ is homotopy equivalent to $S^n$, $\RP^n$, $\CP^{n/2}$, or a lens space.
\end{theorem}

By a theorem of Kobayashi \cite{Kob}, fixed-point set components of isometric actions are totally geodesic submanifolds. In positive curvature, the existence of a totally geodesic submanifold of low codimension imposes strong topological restrictions on the ambient manifold. The Codimension Two Lemma of \cite{KKSS} and the result below it from \cite{AS} are what we will use to exploit this phenomenon in the proof of Theorem \hyperref[main1]{A}. 

\begin{Codim2}\label{codim2}
Let $M^{n+2}$ be a closed, positively curved Riemannian manifold. Suppose $N^n$ is a closed, totally geodesic submanifold of $M$.
Assume one of the following holds:
	\begin{enumerate}[font=\normalfont]
	\item \cite{FRW} $M$ is odd-dimensional, of dimension $\geq 5$; 
	\item \cite{KKSS} The inclusion $N^n \subseteq M^{n+2}$ is $n$-connected; or
	\item  $M^{n+2}$ admits a $\Z_p^2$ action such that $N^n$ is a fixed-point set component of a $\Z_p$ subgroup and such that there is a second $\Z_p$ subgroup whose fixed-point set component $N'$ has codimension  at most $\frac{n+1}{2}$;
	\end{enumerate}
then $N^n$ is homotopy equivalent to $S^n$, $\RP^n$,  $\CP^{\frac n 2}$, or a lens space.
In particular, $M^{n+2}$ is homotopy equivalent to $S^{n+2}$, $\RP^{n+2}$,  $\CP^{\frac{n+2}{2}}$, or a lens space.
\end{Codim2}

\begin{lemma}\cite{AS}\label{combo} Let $M^n$ be a closed,  positively curved Riemannian $n$-manifold, $n\geq 5$. Suppose that $N_1^{n-k_1}$ and $N_2^{n-k_2}$ are two closed, totally geodesic submanifolds of $M$ with $ k_1\leq k_2$, $4k_1\leq n+3$, and $k_1+2k_2\leq n+1$. Then 
\begin{enumerate}
\item If $\widetilde{N_2}$ is a cohomology sphere and $N_2$ has cyclic fundamental group, then  
 $M$ is homotopy equivalent to $S^n$, $\RP^n$, or a lens space; and
 \item If  $\widetilde{N_2}$ is cohomology complex projective space,  $2|k_1$ and $\pi_1(N_2)$ is trivial, then $M$ is homotopy equivalent to $\CP^{n/2}$.
\end{enumerate}

  \end{lemma}

We next recall the Borel formula, which expresses the codimension of a fixed-point component in terms of the codimensions of the fixed-point set components of its corank-one subgroups.

\begin{Borel} \label{borel} \cite{B}
Let $p>2$ be a prime and let $\mathbb{Z}_p^r$ act smoothly on $M$, a Poincaré duality space, with fixed-point set component $F$. Then
$$
\codim (F \subseteq M) \;=\; \sum \codim(F \subseteq F'),
$$
where the sum runs over the fixed-point set components $F'$ of corank one subgroups 
$\mathbb{Z}_p^{r-1} \subseteq \mathbb{Z}_p^{r}$ for which $\dim(F')>\dim(F)$.

These subgroups are precisely the kernels of the irreducible subrepresentations of
the isotropy representation of $\mathbb{Z}_p^r$ on the normal space to $F$. 
In particular, the number of pairwise inequivalent irreducible subrepresentations 
is at least $r$ whenever the action is effective. Moreover, equality holds only if 
the isotropy representation is equivalent to a block diagonal representation of the form
\[
(\eta_1,\dots,\eta_r) \;\longmapsto\;
\operatorname{diag}\bigl(R(\eta_1)^{m_1},\dots,R(\eta_r)^{m_r}\bigr),
\]
where $m_i>0$ denotes the multiplicity of the $i$-th irreducible summand, and
$R(\eta_i)$ is the $2\times 2$ real rotation matrix corresponding to
multiplication by a primitive $p$-th root of unity $\eta_i$. 
Here $R(\eta_i)^{m_i}$ denotes the block diagonal matrix consisting of $m_i$ copies of $R(\eta_i)$.
\end{Borel}

We finally record the following analogue of Observation 3.6 in \cite{KKSS}. Its proof is identical to that of \cite{KKSS}, since it depends only on the Connectedness Lemma and the \hyperref[borel]{Borel formula}, and not on the value of $p$.

\begin{observe}\cite{KKSS}\label{borelobserve}
    Let $F_j^{m_j}$ be the fixed-point set component of a $\Z_p^{r-j}$-action by isometries on a closed, positively curved manifold, and suppose that $r-j \geq 2.$ If the isotropy representation has exactly $r-j$ irreducible subrepresentations, and if the fixed-point set component $F_{j+1}$ containing $F_j$ of some $\Z_p^{r-j-1}$ has the property that the codimension $k_j$ of the inclusion $F_j \subseteq F_{j+1}$ is minimal, then this inclusion is $\dim(F_j)$-connected.  
\end{observe}

\section{The Proof of Theorem A}\label{4}

Before sketching the proof of the main theorem, we begin by proving a lemma giving the existence of two fixed-point set components of low codimension.

\begin{lemma}\label{lemmaA}
Assume $\Z_p^r$ acts effectively and isometrically on a closed, positively
curved manifold $M^n$ with a non-empty fixed-point set, where 
$p \in \{3,5,7,11,13,17,19\}$. 
Suppose
\begin{multline*}
r>
\frac12\Bigl(1-H_p(J_p(\tfrac14))\Bigr)n
+\frac32\log_p n
+\log_p\!\left(\frac{p-1}{2}+\frac1n\right)
-\frac12\log_p 2 \\
\qquad
+\frac12\log_p\!\left(\frac{8(1-J_p(\tfrac14))}{J_p(\tfrac14)}\right)
+\log_p(p-1).
\end{multline*}
Then the following hold:
\begin{enumerate}
\item There exists $\tau_1\in\Z_p^r$ such that $\codim(M^{\tau_1}_x)\le \frac{n+3}{4}$;
\item There exists $\tau_2\in\Z_p^r$ such that $\codim(M^{\tau_2}_x)\le \frac{n+1}{3}$, where $\tau_2\notin\langle\tau_1\rangle$.
\end{enumerate}
\end{lemma}

\begin{proof}
Let $\rho_x:\Z_p^r\rightarrow\Z_p^m$, where $m=\floor{n/2}$, be the isotropy representation at $x$, and set $\mathcal{C}=\rho_x(\Z_p^r)\subseteq\Z_p^m$. Since the action is effective, $\rho_x$ is injective. The proof is identical to that of the corresponding lemma in \cite{AS}, with the only modification being the substitution of the estimate in Theorem \ref{newEB}. In particular, using $\kappa_p(\delta,m)\leq (p-1)m+1$, Theorem \ref{newEB} gives
$$
r\leq (1-H_p(J_p(\delta)))m+\frac12\log_p m+\frac12\log_p\left(\frac{8(1-J_p(\delta))}{J_p(\delta)}\right)+\log_p(p-1)+\log_p((p-1)m+1).
$$
Setting $\delta=\frac14$, the hypothesis of the lemma contradicts this inequality unless there exists a non-trivial element $\tau_1\in\Z_p^r$ such that
$$
\codim(M_x^{\tau_1})\leq\frac{n+3}{4}.
$$
Likewise, restricting $\rho_x$ to a subgroup $\Z_p^{r-1}\subseteq\Z_p^r$ complementary to $\langle\tau_1\rangle$ and setting $\delta=\frac13$ gives, as in \cite{AS}, an element $\tau_2\notin\langle\tau_1\rangle$ such that
$$
\codim(M_x^{\tau_2})\leq\frac{n+1}{3}.
$$
The comparison between the bounds at $\delta=\frac14$ and $\delta=\frac13$ is the same as in \cite{AS}, since the additional term $\log_p((p-1)m+1)$ is independent of $\delta$.
\end{proof}

\begin{remark}
In the geometric applications, we replace $\kappa_p(\delta,m)$ by the simpler estimate $(p-1)m+1$. For $3\leq p\leq19$, $\delta\in\{\frac14,\frac13\}$, and $m=\floor{n/2}$ with $n\geq13$, the resulting loss
$$
\log_p\left(\frac{(p-1)m+1}{\kappa_p(\delta,m)}\right)
$$
is at most $\log_3(15)=2.4649735207\ldots$, attained when $p=3$, $\delta=\frac14$, and $m=7$. We make this relaxation because it substantially simplifies the subsequent inductive arguments while sacrificing only a small additive constant.
\end{remark}

We now recall the statement of Theorem \hyperref[main1]{A} for the convenience of the reader.

\begin{main1}
Assume $\Z_p^r$ acts effectively and isometrically on a closed, positively
curved manifold $M^n$ with a fixed point $x\in M$, where $p \in \{3,5,7,11,13,17,19\}$ and $n \ge 13$. 
Suppose
\begin{multline*}
r>
\frac12\Bigl(1-H_p(J_p(\tfrac14))\Bigr)n
+\frac32\log_p n
+\log_p\!\left(\frac{p-1}{2}+\frac1n\right)
-\frac12\log_p 2 \\
\qquad
+\frac12\log_p\!\left(\frac{8(1-J_p(\tfrac14))}{J_p(\tfrac14)}\right)
+\log_p(p-1)
=: G(n,p).
\end{multline*}
where $H_p$ and $J_p$ are the $p$-ary entropy and Johnson functions defined in Section \ref{2}. Then $M$ is homotopy equivalent to $S^n$, $\RP^n$, $\CP^{n/2}$, or a lens space.
\end{main1}

\begin{proof}
The proof follows that of Theorem B in \cite{AS} and is by total induction on dimension. The result in Theorem \ref{Gh+Zp} provides the base case for the induction.

By Part (1) of Lemma \ref{lemmaA}, there exists a non-trivial $\tau_1\in\Z_p^r$ and a component $x\in N_1^{n-k_1}\subseteq M^{\tau_1}$ such that $k_1\leq\frac{n+3}{4}$ and $N_1$ is of maximal dimension. By maximality, the kernel of the induced $\Z_p^r$-action on $N_1$ is $\Z_p$, and hence $N_1$ admits an effective $\Z_p^{r-1}$-action fixing $x$.

By Part (2) of Lemma \ref{lemmaA}, there exists a non-trivial element $\tau_2\notin\langle\tau_1\rangle$ and a component $x\in N_2^{n-k_2}\subseteq M^{\tau_2}$ such that $k_2\leq\frac{n+1}{3}$ and $N_2$ is of maximal dimension. Again, $N_2$ admits an effective $\Z_p^{r-1}$-action fixing $x$. Without loss of generality, assume $k_1\leq k_2$. Since fixed-point components have even codimension, there are two cases.

\noindent\textbf{Case 1: $k_2\geq4$.} Using $r>G(n,p)$ and $k_2\geq4$, a computation gives
$$
r-1>G(n-k_2,p).
$$
The induction hypothesis therefore implies that $N_2$ is homotopy equivalent to $S^{n-k_2}$, $\RP^{n-k_2}$, $\CP^{(n-k_2)/2}$, or a lens space. The remainder of the argument is identical to Case 1 in the proof of Theorem B in \cite{AS}, and involves invoking Lemma \ref{combo} to obtain the corresponding homotopy classification of $M$.

\noindent\textbf{Case 2: $k_2=2$.} Since $k_1\leq k_2$ and both codimensions are positive and even, we have $k_1=k_2=2$. The argument is then identical to Case 2 in the proof of Theorem B in \cite{AS}. Using the \hyperref[borel]{Borel Formula}, Observation \ref{borelobserve}, and the \hyperref[codim2]{Codimension Two Lemma}, it follows that $M$ is homotopy equivalent to $S^n$, $\RP^n$, $\CP^{n/2}$, or a lens space.

This completes the induction and completes the proof.
\end{proof}

\section{Declarations}

\subsection{Funding}

The author was partially supported by Catherine Searle's NSF grants DMS-2204324 and DMS-2506633, and her Simons’ Foundation Travel Grant for Mathematicians \#SFI-MPS-TSM-00012804 (2025–2030).

\subsection{Competing Interests}

The author has no competing interests to declare that are relevant to the content of this article.

\end{document}